\documentclass[12pt,twoside]{amsart}
\usepackage{amssymb}
\usepackage{amscd}
\usepackage{amsmath}
\usepackage[all]{xy}

\title[Vanishing and semipositivity theorems for slc pairs]
{Vanishing and semipositivity theorems for semi-log canonical pairs}
\author{Osamu Fujino}
\date{2018/2/15, version 0.10}
\address{Department of Mathematics, Graduate School of Science,
Osaka University, Toyonaka, Osaka 560-0043, Japan}
\email{fujino@math.sci.osaka-u.ac.jp}

\newcommand{\Sing}[0]{\operatorname{Sing}}
\newcommand{\Supp}[0]{\operatorname{Supp}}
\newcommand{\Exc}[0]{\operatorname{Exc}}
\newcommand{\codim}[0]{\operatorname{codim}}

\newtheorem{thm}{Theorem}[section]
\newtheorem{lem}[thm]{Lemma}
\newtheorem{cor}[thm]{Corollary}

\theoremstyle{definition}
\newtheorem{defn}[thm]{Definition}
\newtheorem{rem}[thm]{Remark}
\newtheorem*{ack}{Acknowledgments}
\newtheorem{say}[thm]{}

\makeatletter
    
    \@addtoreset{equation}{section}
\makeatother

\begin{document}
\bibliographystyle{amsalpha+}

\maketitle

\begin{abstract}
We prove an effective vanishing theorem for direct images of 
log pluricanonical bundles of projective 
semi-log canonical pairs. 
As an application, we obtain a semipositivity theorem 
for direct images of relative log pluricanonical bundles 
of projective semi-log canonical pairs over curves, 
which implies the projectivity of the moduli spaces of 
stable varieties. 
It is worth mentioning that we 
do not use the theory of variation of (mixed) Hodge structure. 
\end{abstract}

\tableofcontents
\section{Introduction}\label{f-sec1}

In this paper, we establish some vanishing theorems for 
semi-log canonical 
pairs and prove some semipositivity theorems 
for semi-log canonical pairs as applications without using the theory 
of graded polarizable admissible variation of mixed Hodge structure. 

First we prove an effective vanishing theorem for direct images of 
log pluricanonical bundles of projective 
semi-log canonical pairs, which is 
a generalization of \cite[Theorem 1.7]{popa-schnell}. 

\begin{thm}[Effective vanishing theorem]\label{f-thm1.1}
Let $(X, \Delta)$ be a projective semi-log canonical 
pair and let $f:X\to Y$ be a surjective morphism 
onto an $n$-dimensional projective variety $Y$. 
Let $D$ be a Cartier divisor on $X$ such that 
$D\sim _{\mathbb R} k(K_X+\Delta+f^*H)$ for some 
positive integer $k$, 
where $H$ is an ample $\mathbb R$-divisor on $Y$. 
Let $L$ be an ample Cartier divisor on $Y$ such that 
$|L|$ is free. 
Assume that $\mathcal O_X(D)$ is $f$-generated. 
Then $$H^i(Y, f_*\mathcal O_X(D)\otimes \mathcal O_Y(lL))=0$$ for 
every $i>0$ and every $l\geq (k-1) (n+1-t) -t +1$, 
where $t=\sup \{ s \, | \, H-sL\ \text{is ample}\}$. 
Therefore, by the Castelnuovo--Mumford regularity, 
$f_*\mathcal O_X(D)\otimes \mathcal O_Y(lL)$ is 
globally generated for every $l\geq (k-1) (n+1-t)-t+1+n$.  
\end{thm}

We note that Theorem \ref{f-thm1.1} is a 
consequence of the Koll\'ar--Ohsawa type 
vanishing theorem for 
semi-log canonical pairs (see Theorem \ref{f-thm3.1} below). 
When $(X, \Delta)$ is log canonical, that is, 
$X$ is normal, in Theorem \ref{f-thm1.1}, 
Mihnea Popa and Christian Schnell (see \cite{popa-schnell}) 
proved that Theorem \ref{f-thm1.1} holds true without assuming that 
$\mathcal O_X(D)$ is $f$-generated. 
Therefore, Theorem \ref{f-thm1.1} is much weaker than 
\cite[Theorem 1.7]{popa-schnell} when $(X, \Delta)$ is log canonical.  
However, it is sufficiently powerful. 

Next we prove a semipositivity theorem for direct images of relative log pluricanonical 
bundles of projective semi-log canonical pairs over curves 
as an application of 
Theorem \ref{f-thm1.1}, which 
is a special case of \cite[Theorem 1.11]{fujino-semipositive}. 
Note that the results in \cite{fujino-semipositive} heavily depend 
on the theory of graded polarizable admissible variation of mixed Hodge structure 
(see \cite{fujino-fujisawa} and \cite{ffs}). 
Therefore, the reader may feel that 
Theorem \ref{f-thm1.2} is more 
accessible than \cite[Theorem 1.11]{fujino-semipositive}. 
We strongly recommend the reader to compare 
Theorem \ref{f-thm1.2} with \cite[Theorem 1.11]{fujino-semipositive}. 

\begin{thm}[Semipositivity theorem]\label{f-thm1.2}
Let $(X, \Delta)$ be a projective semi-log canonical pair and let $f:X\to Y$ be 
a flat morphism onto a smooth projective curve $Y$ such that 
\begin{itemize}
\item[(i)] $\Supp \Delta$ avoids the generic and codimension one 
singular points of every fiber of $f$, and 
\item[(ii)] $(X_y, \Delta_y)$ is a semi-log canonical pair for 
every $y\in Y$. 
\end{itemize}
Assume that $\mathcal O_X(k(K_X+\Delta))$ is locally free and 
$f$-generated for some positive integer $k$. 
Then $f_*\mathcal O_X(k(K_{X/Y}+\Delta))$ is a nef locally free sheaf. 
\end{thm}

Although Theorem \ref{f-thm1.2} is a very special case 
of \cite[Theorem 1.11]{fujino-semipositive}, 
it seems to be sufficient for most geometric 
applications (see \cite{fujino-semipositive}, 
\cite{patakfalvi2}, \cite[Lemma 7.7]{kovacs-patakfalvi}, \cite
[Theorem 2.13]{patakfalvi-xu}, \cite{ascher-turchet}, etc.) 
By Koll\'ar's projectivity criterion (see \cite{kollar-projectivity}) and 
Theorem \ref{f-thm1.2}, 
we can easily obtain: 

\begin{thm}[{\cite[Theorem 1.1]{fujino-semipositive}}]\label{f-thm1.3} 
Every complete subspace of the coarse moduli 
space of stable varieties is {\em{projective}}. 
\end{thm}

In this paper, we only sketch the proof of Theorem \ref{f-thm1.3} for the 
reader's convenience. 
We recommend the reader to 
see \cite{kollar-projectivity} and \cite{fujino-semipositive} for the details 
of Theorem \ref{f-thm1.3} (see also \cite{kovacs-patakfalvi}). 

Finally, we give a proof of \cite[Theorem 1.9]{fujino-semipositive}, 
which is called the basic semipositivity theorem in \cite{fujino-semipositive},  
based on the Koll\'ar--Ohsawa type vanishing theorem for semi-log canonical 
pairs (see Theorem \ref{f-thm3.1} and Remark \ref{f-rem4.2}). 

\begin{thm}[Basic semipositivity theorem 
{\cite[Theorem 1.9]{fujino-semipositive}}]\label{f-thm1.4} 
Let $(X, D)$ be a simple normal crossing pair 
such that $D$ is reduced. 
Let $f:X\to C$ be a projective surjective morphism 
onto a smooth projective curve $C$. 
Assume that every stratum of $X$ is dominant onto $C$. 
Then $f_*\omega_{X/C}(D)$ is nef. 
\end{thm}

It is worth mentioning that all the semipositivity theorems in \cite{fujino-semipositive} 
follow from \cite[Theorem 1.9]{fujino-semipositive}, 
that is, Theorem \ref{f-thm1.4}. 
Therefore, by replacing the proof of 
\cite[Theorem 1.9]{fujino-semipositive} 
with the proof of Theorem \ref{f-thm1.4} 
given in this paper, 
the paper \cite{fujino-semipositive} 
becomes independent of the theory of graded polarizable 
admissible variation of mixed Hodge structure. 
In particular, we see that the projectivity of moduli spaces of stable varieties and 
pairs (see \cite{fujino-semipositive} and \cite{kovacs-patakfalvi}) 
can be established without appealing 
to the theory of variation of (mixed) Hodge structure 
(see also Theorem \ref{f-thm1.3}). 
We note that the main ingredient of 
this paper is the vanishing theorem for simple normal crossing 
pairs (see \cite{fujino-vanishing}, \cite{fujino-injectivity}, 
and \cite{fujino-foundation}), 
which comes from the theory of mixed Hodge structure on cohomology 
with compact support. 
We also note that this paper does not supersede \cite{fujino-semipositive} 
but will complement \cite{fujino-semipositive}. 

\begin{say}[Historical comments]\label{f-say1.5}
In 2012, I wrote and submitted \cite{fujino-semipositive}. 
Unfortunately, some referee kept it for a very long time without 
making decisions. 
In 2017, 
the editor changed the referee. 
Then the process became a usual one. 
Now \cite{fujino-semipositive} was accepted for publication. 
By the way, I obtained the results of this paper in 2015. 
I think that they make \cite{fujino-semipositive} more accessible. 
However, since the referee was keeping \cite{fujino-semipositive}, 
I could not revise \cite{fujino-semipositive}. 
So I wrote this paper separately. 
Anyway, I recommend the reader to read \cite{fujino-semipositive} too. 
\end{say}

\begin{ack}
The author was partially supported by JSPS KAKENHI Grant Numbers JP16H03925, 
JP16H06337. 
When he wrote the original version of this paper in 2015, 
he was partially supported by Grant-in-Aid for Young Scientists (A) 24684002 from 
JSPS. Some parts of this work were completed while the author 
was visiting University of Utah to attend AMS Summer Institute 
in Algebraic Geometry. 
\end{ack}

We will work over $\mathbb C$, the complex number field, throughout this paper.
Note that, by the Lefschetz principle,
all the results in 
this paper hold over any algebraically closed field $k$ of characteristic zero.
For the standard 
notations and conventions of the 
log minimal model program, 
see \cite{fujino-fundamental} and \cite{fujino-foundation}. 
In this paper, a {\em{variety}} means a separated reduced scheme of 
finite type over $\mathbb C$. 

\section{Preliminaries}\label{f-sec2} 

In this section, we collect some basic definitions and results. 
Note that we are mainly interested in non-normal reducible equidimensional 
varieties. 

We need the notion of {\em{simple normal crossing pairs}} for 
various vanishing theorems (see, for example, 
\cite{fujino-vanishing}, \cite{fujino-injectivity}, 
\cite{fujino-foundation}, and Theorem \ref{f-thm3.1} below). 
We note that a simple normal crossing pair is sometimes
called a {\em{semi-snc}} pair in the 
literature (see \cite[Definition 1.1]{bierstone-verapacheco} 
and \cite[Definition 1.10]{kollar-book}).

\begin{defn}[Simple normal crossing pairs]\label{f-def2.1}
We say that the pair $(X, D)$ is {\em{simple normal crossing}} at
a point $a\in X$ if $X$ has a 
Zariski open neighborhood $U$ of $a$ that can be embedded in a smooth
variety
$Y$,
where $Y$ has regular system of parameters 
$(x_1, \cdots, x_p, y_1, \cdots, y_r)$ at
$a=0$ in which $U$ is defined by a monomial equation
$$
x_1\cdots x_p=0
$$
and $$
D=\sum _{i=1}^r \alpha_i(y_i=0)|_U, \quad  \alpha_i\in \mathbb R.
$$
We say that $(X, D)$ is a {\em{simple normal crossing pair}} if it is 
simple normal crossing at every point of $X$. 
We sometimes 
say that $D$ is a {\em{simple normal crossing divisor}} on $X$ 
if $(X, D)$ is a simple normal crossing pair and $D$ is reduced. 
\end{defn}

\begin{say}[$\mathbb Q$-divisors and 
$\mathbb R$-divisors]\label{f-say2.2} 
Let $D$ be an $\mathbb R$-divisor (resp.~a $\mathbb Q$-divisor)  
on an equidimensional variety $X$, that is, 
$D$ is a finite formal $\mathbb R$-linear (resp.~$\mathbb Q$-linear) 
combination 
$$D=\sum _i d_i D_i$$ of irreducible 
reduced subschemes $D_i$ of codimension one such that 
$D_i\ne D_j$ for $i\ne j$. 
We define the {\em{round-up}} 
$\lceil D\rceil =\sum _i \lceil d_i \rceil D_i$, where for 
every real number $x$, $\lceil x\rceil$ is the integer 
defined by $x\leq  \lceil x \rceil <x+1$. 
We set 
$$D^{<1}=\sum _{d_i<1}d_i D_i. 
$$
We say that $D$ is a {\em{boundary}} (resp.~{\em{subboundary}}) 
$\mathbb R$-divisor if $0\leq d_i \leq 1$ (resp.~$d_i\leq 1$) for every $i$. 
\end{say}

\begin{say}[$\mathbb R$-linear equivalence]\label{f-say2.3} 
Let $B_1$ and $B_2$ be two $\mathbb R$-Cartier divisors 
on a variety $X$. Then $B_1\sim _{\mathbb R} B_2$ means that 
$B_1$ is {\em{$\mathbb R$-linearly equivalent}} to 
$B_2$. 
\end{say}

Let us recall the definition of {\em{strata}} of simple normal crossing pairs. 

\begin{defn}[Stratum]\label{f-def2.4}
Let $(X, D)$ be a simple normal crossing pair such that 
$D$ is a boundary $\mathbb R$-divisor. 
Let $\nu:X^\nu\to X$ be the normalization. 
We put $K_{X^\nu}+\Theta=\nu^*(K_X+D)$, that is, 
$\Theta$ is the sum of the inverse images of $D$ and the singular locus of 
$X$. 
A {\em{stratum}} of $(X, D)$ is an irreducible component 
of $X$ or the $\nu$-image of a log canonical 
center of $(X^\nu, \Theta)$. We note that 
$(X^\nu, \Theta)$ is log canonical since $(X, D)$ is a simple normal crossing 
pair and $D$ is a boundary $\mathbb R$-divisor. 
\end{defn}

For the reader's convenience, we recall the definition 
of {\em{semi-log canonical pairs}}. 

\begin{defn}[Semi-log canonical pairs]\label{f-def2.5} 
Let $X$ be an equidimensional variety which satisfies Serre's 
$S_2$ condition and is normal crossing in codimension one. 
Let $\Delta$ be an effective $\mathbb R$-divisor 
on $X$ such that no irreducible component of $\Supp \Delta$ is contained 
in the singular locus of $X$. 
The pair $(X, \Delta)$ is called a {\em{semi-log canonical pair}} 
(an {\em{slc pair}}, for short) if 
\begin{itemize}
\item[(1)] $K_X+\Delta$ is $\mathbb R$-Cartier, and 
\item[(2)] $(X^\nu, \Theta)$ is log canonical, where 
$\nu:X^\nu\to X$ is the normalization and 
$K_{X^\nu}+\Theta=\nu^*(K_X+\Delta)$, that is, $\Theta$ is 
the sum of the inverse images of $\Delta$ and the 
conductor of $X$. 
\end{itemize}
If $(X, 0)$ is a semi-log canonical pair, then we simply say that 
$X$ is a {\em{semi-log canonical variety}} or $X$ has only {\em{semi-log canonical 
singularities}}. 
\end{defn}

For the details of semi-log canonical pairs and 
the basic notations, see \cite{fujino-slc} and \cite{kollar-book}. 

\medskip 

In the recent literature, an equidimensional variety which is 
normal crossing in codimension one and satisfies Serre's $S_2$ condition 
is sometimes said to be {\em{demi-normal}}. 

\begin{defn}[Koll\'ar]\label{f-def2.6}
An equidimensional variety $X$ is said to be {\em{demi-normal}} 
if $X$ satisfies Serre's $S_2$ condition and 
is normal crossing in codimension one. 
\end{defn}

By definition, if $(X, \Delta)$ is a semi-log canonical pair, 
then $X$ is demi-normal. 
For the details of divisors and divisorial sheaves on demi-normal varieties 
and semi-log canonical pairs, 
see \cite[Section 5.1]{kollar-book}. 

\begin{say}[{cf.~\cite{patakfalvi1}}]\label{f-say2.7} 
For a complex $\mathcal C^{\bullet}$ of sheaves, 
$h^i(\mathcal C^{\bullet})$ is the $i$-th cohomology sheaf of $\mathcal C^{\bullet}$. 
For a morphism $f:X\to Y$ between 
separated schemes of finite type over $\mathbb C$, 
we put $\omega^{\bullet}_{X/Y}=f^{!}\mathcal O_Y$, where 
$f^!$ is the functor obtained in \cite[Chapter VII, Corollary 3.4 (a)]{hartshorne1} 
(see also \cite{conrad}). 
If $f$ has equidimensional fibers of dimension $n$, 
then we put $\omega_{X/Y}=h^{-n} (\omega^{\bullet}_{X/Y})$ and call it 
the {\em{relative canonical sheaf}} of $f:X\to Y$. 
We recommend the reader to see \cite[Section 3]{patakfalvi1} 
for some basic properties of (relative) canonical sheaves 
and base change properties. 
\end{say}

Although the following definition is slightly 
different from \cite[Definition 3.3]{kovacs-patakfalvi}, there are no problems 
since almost all varieties we treat in this paper satisfy Serre's $S_2$ condition. 

\begin{defn}\label{f-def2.8}
Let $Z$ be an equidimensional variety. 
A {\em{big open subset}} $U$ of $Z$ is a Zariski open 
subset $U\subset Z$ such that $\mathrm{codim}_Z(Z\setminus U)\geq 2$. 
\end{defn}

We discuss the {\em{divisorial pull-backs}} of $\mathbb Q$-divisors and 
$\mathbb R$-divisors 
under some 
suitable assumptions (cf.~\cite[Notation 3.7]{kovacs-patakfalvi}). 

\begin{say}\label{f-say2.9}
Let $f:X\to Y$ be a flat projective morphism from a 
demi-normal variety $X$ to a smooth 
curve $Y$. We further assume that every fiber $X_y$ of $f$ is demi-normal. 
Let $D$ be a $\mathbb Q$-divisor on $X$ that 
avoids the generic and codimension one singular points of 
every fiber of $f$. 
We will denote by $D_y$ the {\em{divisorial pull-back}} of $D$ to $X_y$, which is defined 
as follows: As $D$ avoids the singular codimension one 
points of $X_y$, 
we can take a big open subset $U\subset X$ such that 
$D|_U$ is $\mathbb Q$-Cartier and 
that $U_y=U|_{X_y}$ is also a big open subset of $X_y$. 
We define $D_y$ to be the unique $\mathbb Q$-divisor on $X_y$ whose 
restriction to $U_y$ is $(D|_U)|_{U_y}$. 
By adjunction, we have $(K_X+X_y)|_{X_y}=K_{X_y}$. 
Note that $X_y$ is a Cartier divisor on $X$ since $Y$ is a smooth curve 
and that $X_y$ is Gorenstein in codimension one since 
$X_y$ is demi-normal. 
Therefore, in Theorem \ref{f-thm1.2}, we have 
$(K_{X/Y}+\Delta)|_{X_y}=K_{X_y}+\Delta_y$. 
Moreover, if $m(K_X+\Delta)$ is Cartier for some positive integer $m$ in 
Theorem \ref{f-thm1.2}, then $\mathcal O_X(m(K_{X/Y}+\Delta))|_{X_y}
\simeq \mathcal O_{X_y}(m(K_{X_y}+\Delta_y))$. 
\end{say} 

\begin{say}\label{f-say2.10} Let $f:X\to Y$ be a flat morphism 
between demi-normal varieties. 
Let $D$ be an $\mathbb R$-divisor on $Y$ that avoids the codimension 
one singular points of $Y$. 
Then we will denote by $f^{-1}D$ the {\em{divisorial 
pull-back}} of $D$ to $X$, which is defined as follows: 
As $D$ avoids the singular codimension one points of $Y$, there is 
a big open subset $U\subset Y$ such that 
$D|_U$ is $\mathbb R$-Cartier. We define $f^{-1}D$ to be the unique 
$\mathbb R$-divisor on $X$ whose restriction to $f^{-1}(U)$ is $f^*(D|_U)$. 
\end{say}

The following two lemmas are well-known and easy to check. 
So we leave the proof as exercises for the reader. 
We will use these lemmas in the proof of Lemma \ref{f-lem2.13}. 

\begin{lem}\label{f-lem2.11}
Let $(X, \Delta)$ be a semi-log canonical pair. 
Let $\Supp \Delta=\sum _i B_i$ be the irreducible 
decomposition. 
We define a finite-dimensional 
$\mathbb R$-vector space $V=\bigoplus _i \mathbb R B_i$. 
Then 
we see 
that 
$$
\mathcal L=\{ D\in V \, |\, (X, D) \ {\text{is semi-log canonical}}\}
$$ 
is a rational polytope in $V$. 
Therefore, we can write 
$$
K_X+\Delta=\sum _{i=1}^k r_i (K_X+D_i), 
$$ 
where 
\begin{itemize}
\item[(i)] $D_i \in \mathcal L$ for every $i$; 
\item[(ii)] $D_i$ is a $\mathbb Q$-divisor for every $i$; 
\item[(iii)] $0<r_i < 1$, $r_i\in \mathbb R$ for every $i$, and 
$\sum _{i=1}^k r_i =1$.  
\end{itemize}
\end{lem}

\begin{lem}\label{f-lem2.12} 
Let $(V_i, D_i)$ be a log canonical pair such that 
$K_{V_i}+D_i$ is 
$\mathbb Q$-Cartier for $i=1, 2$. 
We put $V=V_1\times V_2$ and $D=p_1^{-1}D_1+p_2^{-1}D_2$, 
where $p_i: V\to V_i$ is the $i$-th projection for $i=1, 2$. 
Then $(V, D)$ is log canonical. 
\end{lem}

We will use the following lemma, which 
seems to be well-known to the experts, in the proof of Theorem \ref{f-thm1.2}. 

\begin{lem}\label{f-lem2.13}
Let $(X_i, \Delta_i)$ be a semi-log canonical pair for $i=1, 2$. 
We put $X=X_1\times X_2$ and $\Delta=p_1^{-1}\Delta_1+p_2^{-1}\Delta_2$, 
where $p_i: X\to X_i$ is the $i$-th projection for $i=1, 2$. 
Then $(X, \Delta)$ is a semi-log canonical pair as well. 
\end{lem}

\begin{proof}
By Lemma \ref{f-lem2.11}, 
we may assume that 
$\Delta_1$ and $\Delta_2$ are $\mathbb Q$-divisors on $X_1$ and 
$X_2$ respectively. We see that 
$X=X_1\times X_2$ satisfies Serre's $S_2$ condition and is normal crossing in 
codimension one. We take a positive integer $m$ such that 
$m(K_{X_1}+\Delta_1)$ and $m(K_{X_2}+\Delta_2)$ are Cartier. 
Then we have 
\begin{equation}
\mathcal O_X(m(K_X+\Delta))\simeq p_1^*\mathcal O_{X_1}
(m(K_{X_1}+\Delta_1))\otimes p_2^*
\mathcal O_{X_2}(m(K_{X_2}+\Delta_2)). 
\end{equation} 
Thus $K_X+\Delta$ is $\mathbb Q$-Cartier. 
Let $\nu_i: X_i^\nu\to X_i$ be the normalization. 
We put $K_{X_i^\nu}+\Theta_i=\nu_i^*(K_{X_i}+\Delta_i)$ as in Definition \ref{f-def2.5} for 
$i=1, 2$. 
By definition, $(X_i^\nu, \Theta_i)$ is log canonical 
for $i=1, 2$. 
We note that 
$\nu=\nu_1\times \nu_2: X^\nu=X_1^\nu\times X_2^\nu\to X=X_1\times X_2$ is 
the normalization and $K_{X^\nu}+\Theta=\nu^*(K_X+\Delta)$, where 
$\Theta=q_1^{-1}\Theta_1+q_2^{-1}\Theta_2$ such that 
$q_i: X^\nu\to X_i^\nu$ is the 
$i$-th projection for $i=1, 2$. 
By Lemma \ref{f-lem2.12}, $(X^\nu, \Theta)$ is log canonical. 
Therefore, $(X, \Delta)$ is semi-log canonical. 
\end{proof}

We will need the following definition in Section \ref{f-sec4}. 

\begin{defn}\label{f-def2.14}
Let $\mathcal F$ be a coherent sheaf on a projective 
variety $X$. 
If the natural map 
$$
H^0(X, \mathcal F)\otimes \mathcal O_X\to \mathcal F
$$ 
is generically surjective, then 
$\mathcal F$ is said to be {\em{generically globally generated}}. 
\end{defn}

We close this section with the definition of {\em{nef locally free sheaves}}. 

\begin{defn}[Nef locally free sheaves]\label{f-def2.15} 
A locally free sheaf $\mathcal E$ of finite rank on a 
complete variety $X$ is {\em{nef}} if the following equivalent 
conditions are satisfied: 
\begin{itemize}
\item[(i)] $\mathcal E=0$ or $\mathcal O_{\mathbb P_X(\mathcal 
E)}(1)$ is nef on $\mathbb P_X(\mathcal E)$. 
\item[(ii)] For every map from a smooth projective curve 
$f:C\to X$, every quotient line bundle of $f^*\mathcal E$ has non-negative 
degree. 
\end{itemize}
A nef locally free sheaf was originally called a ({\em{numerically}}) 
{\em{semipositive}} sheaf in the literature. 
\end{defn}

\section{Vanishing and semipositivity theorems}\label{f-sec3} 

Let us start the following vanishing theorem for semi-log canonical 
pairs. 

\begin{thm}[Vanishing theorem]\label{f-thm3.1}
Let $(X, \Delta)$ be a projective semi-log canonical 
pair and let $f:X\to Y$ be a surjective morphism 
onto a projective variety $Y$. 
Let $D$ be a Cartier divisor on $X$ such that 
$D-(K_X+\Delta)\sim _{\mathbb R}f^*H$ for some ample $\mathbb R$-divisor 
$H$ on $Y$. 
Then $H^i(Y, f_*\mathcal O_X(D))=0$ for every $i>0$.  
\end{thm}

We can see Theorem \ref{f-thm3.1} as a generalization of the Koll\'ar--Ohsawa 
vanishing theorem for projective semi-log canonical pairs 
(see \cite[Theorem 3.1]{ohsawa} and 
\cite[Theorem 2.1]{kollar-higher}). 
We note that both the vanishing theorems 
(\cite[Theorem 3.1]{ohsawa} and \cite[Theorem 2.1]{kollar-higher}) 
are now special cases of \cite[Theorem 1.3]{matsumura}. 
Theorem \ref{f-thm3.1} is a key ingredient of the proof of 
Theorem \ref{f-thm1.1}. It follows from the theory of 
mixed Hodge structure on cohomology with compact support 
(see \cite{fujino-vanishing} and \cite{fujino-foundation}). 

\begin{proof}
Let $\pi:\widetilde X\to X$ be a natural 
double cover due to Koll\'ar (see \cite[5.23 (A natural double cover)]{kollar-book}). 
Then $\mathcal O_X(D)$ is a direct summand of $\pi_*\mathcal O_{\widetilde X} 
(\pi^*D)$. Therefore, by replacing $X$ and $D$ with $\widetilde X$ and $\pi^*D$, 
respectively, we may assume that $X$ is simple normal crossing in codimension one. 
Let $U$ be the largest Zariski open subset of $X$ where 
$(U, \Delta|_U)$ is a simple normal crossing pair. 
Then $\mathrm{codim}_X(X\setminus U)\geq 2$. 
By the theorem of Bierstone--Vera Pacheco (see 
\cite[Theorem 1.4]{bierstone-verapacheco}), there 
exists a birational morphism 
$g:Z\to X$ from a projective simple normal crossing variety $Z$ such that 
$g$ is an isomorphism over $U$, 
$g$ maps $\Sing Z$ birationally onto the closure of 
$\Sing U$ in $X$, and that 
$K_Z+\Delta_Z=g^*(K_X+\Delta)$, 
where $\Delta_Z$ is a subboundary $\mathbb R$-divisor such that 
$\Supp \Delta_Z$ is a simple normal crossing divisor on $Z$.  
We put $E=\lceil -\Delta^{<1}_Z\rceil$. 
Then $E$ is an effective $g$-exceptional Cartier divisor. 
By the definition of $E$, $\Delta_Z+E$ is a boundary $\mathbb R$-Cartier 
$\mathbb R$-divisor on $Z$ such that $\Supp (\Delta_Z+E)$ is a simple 
normal crossing divisor on $Z$. 
In particular, $(Z, \Delta_Z+E)$ is a simple normal crossing pair. 
Since we have 
$$
g^*D+E-(K_Z+\Delta_Z+E)\sim _{\mathbb R} g^*f^*H, 
$$ 
we obtain 
$$
H^i(Y, (f\circ g)_*\mathcal O_Z(g^*D+E))=0  
$$ 
for every $i>0$ (see \cite[Theorem 1.1]{fujino-vanishing} and \cite{fujino-foundation}). 
This implies that $H^i(Y, f_*\mathcal O_X(D))=0$ for every $i>0$ because 
$g_*\mathcal O_Z(g^*D+E)\simeq \mathcal O_X(D)$. 
\end{proof}

By the Castelnuovo--Mumford regularity and 
Theorem \ref{f-thm3.1}, 
we have: 

\begin{cor}\label{f-cor3.2}
Let $(X, \Delta)$ be a projective semi-log canonical pair 
such that $K_X+\Delta$ is Cartier and let $f:X\to Y$ be a surjective 
morphism onto a projective variety $Y$ with $\dim Y=n$. 
Let $L$ be an ample Cartier divisor on $Y$ such that $|L|$ is free. 
Then $f_*\mathcal O_X(K_X+\Delta)\otimes \mathcal O_Y(lL)$ is globally generated 
for every $l\geq n+1$. 
\end{cor} 
\begin{proof}
We put $D=K_X+\Delta+f^*(lL)$ with 
$l\geq n+1$. 
Then we have that 
$H^i(Y, f_*\mathcal O_X(D)\otimes \mathcal O_Y(-iL))=0$ for 
every $i>0$ by Theorem \ref{f-thm3.1}. 
Therefore, by the Castelnuovo--Mumford regularity, 
we obtain that $f_*\mathcal O_X(K_X+\Delta)\otimes \mathcal O_Y(lL)$ is 
globally generated for every $l\geq n+1$. 
\end{proof}

We note that we will use Corollary \ref{f-cor3.2} in the proof of Theorem \ref{f-thm1.4} in 
Section \ref{f-sec4}. 

\medskip

Let us start the proof of Theorem \ref{f-thm1.1}, which is a clever application of 
Theorem \ref{f-thm3.1}. 

\begin{proof}[Proof of Theorem \ref{f-thm1.1}]  
We closely follow the 
proof of \cite[Theorem 1.7]{popa-schnell}. 
Since $L$ is ample, there exists the smallest integer $m\geq 0$ such that 
$f_*\mathcal O_X(D)\otimes \mathcal O_Y(mL)$ is globally generated. 
Since $f^*f_*\mathcal O_X(D)\to \mathcal O_X(D)$ is surjective by assumption, 
$\mathcal O_X(D)\otimes f^*\mathcal O_Y(mL)$ is globally generated as well. 
Let $B$ be a general member of 
the free linear system 
$|\mathcal O_X(D)\otimes f^*\mathcal O_Y(mL)|$. 
Then 
$$
k(K_X+\Delta+f^*H)+mf^*L\sim _{\mathbb R} B. 
$$ 
Therefore, we have  
$$
(k-1)(K_X+\Delta+f^*H)\sim _{\mathbb R} \frac{k-1}{k}B-\frac{k-1}{k} 
mf^*L. 
$$
For any integer $l$, we can write 
$$
D+lf^*L\sim _{\mathbb R} K_X+\Delta+\frac{k-1}{k} B 
+f^*\left(H+\left(l-\frac{k-1}{k}m\right)L\right). 
$$ 
We note that the $\mathbb R$-divisor 
$$
H+\left(l-\frac{k-1}{k}m\right)L
$$ 
on $Y$ is ample if $l+t-\frac{k-1}{k}m>0$. 
We also note that $(X, \Delta+\frac{k-1}{k}B)$ is semi-log canonical since $B$ is 
a general member of 
the free linear system 
$|\mathcal O_X(D)\otimes f^*\mathcal O_Y(mL)|$. 
Thus we obtain that 
$$
H^i(Y, f_*\mathcal O_X(D)\otimes \mathcal O_Y(lL))=0
$$ 
for all $i>0$ and $l>\frac{k-1}{k}m -t$ by Theorem \ref{f-thm3.1}. 
Therefore, for every $l>\frac{k-1}{k} m -t +n$, 
we have 
$$
H^i(Y, f_*\mathcal O_X(D)\otimes \mathcal O_X(lL)\otimes \mathcal O_Y(-iL))=0
$$ 
for every $i>0$. 
Hence, $f_*\mathcal O_X(D)\otimes \mathcal O_Y(lL)$ 
is globally generated for $l>\frac{k-1}{k} m -t +n$ 
by the Castelnuovo--Mumford regularity. 
Given our minimal choice of $m$, 
we conclude that for the smallest integer $l_0$ which is 
greater than $\frac{k-1}{k}m-t$ we have $m\leq l_0+n$. 
This implies that 
$$
m\leq l_0 +n \leq \frac{k-1}{k}  m +n +1-t, 
$$ 
which is equivalent to $m\leq k(n+1-t)$. 
Thus, 
$$
H^i(Y, f_*\mathcal O_X(D)\otimes \mathcal O_Y(lL))=0
$$ 
for every $i>0$ and $l\geq (k-1)(n+1-t) -t+1$. 
\end{proof}

As a direct consequence of Theorem \ref{f-thm1.1}, we have: 

\begin{cor}\label{f-cor3.3}
Let $(X, \Delta)$ be a projective semi-log canonical 
pair and let $f:X\to Y$ be a surjective morphism 
onto an $n$-dimensional projective variety $Y$. 
Assume that $\mathcal O_X(k(K_X+\Delta))$ is locally free and 
$f$-generated for some positive integer $k$. 
Let $L$ be an ample Cartier divisor on $Y$ such that 
$|L|$ is free. 
Then $$H^i(Y, f_*\mathcal O_X(k(K_X+\Delta))\otimes \mathcal O_Y(lL))=0$$ for 
every $i>0$ and every $l\geq k(n+1)-n$. 
Therefore, by the Castelnuovo--Mumford 
regularity, 
$f_*\mathcal O_X(k(K_X+\Delta))\otimes \mathcal O_Y(lL)$ is 
globally generated for every $l\geq k(n+1)$.  
\end{cor}

\begin{proof}
We put $D=k(K_X+\Delta+f^*L)$ and apply Theorem \ref{f-thm1.1}. 
Then we obtain 
$$
H^i(Y, f_*\mathcal O_X(D)\otimes \mathcal O_Y((l-k)L))=0
$$ 
for every $i>0$ and every 
$l-k\geq (k-1)n$, equivalently, $l\geq k(n+1)-n$. 
\end{proof}

Corollary \ref{f-cor3.3} will play a crucial role in the proof of Theorem \ref{f-thm1.2}. 
Let us prove Theorem \ref{f-thm1.2}. The idea of 
the proof of Theorem \ref{f-thm1.2} is the same as \cite[Section 5]{fujino-direct} 
(see also \cite{fujino-co} and 
\cite[Subsection 3.1]{fujino-zucker65}). 

\begin{proof}[Proof of Theorem \ref{f-thm1.2}]
Let $s$ be an arbitrary positive integer. 
Let $$X^{(s)}=\underbrace{X\times _Y X\times _Y \cdots \times _Y X}_{s}$$ be 
the $s$-fold fiber product of $X$ over $Y$ and let 
$f^{(s)}: X^{(s)}\to Y$ be the induced natural map. 
Let $p_i$ be the $i$-th projection $X^{(s)}\to X$ for $1\leq i\leq s$. 
We put $\Delta_{X^{(s)}}=\sum _{i=1}^s p_i^{-1}\Delta$. 
Then we have:  
\begin{equation}\label{eq3.1}
\mathcal O_{X^{(s)}}(k(K_{X^{(s)}/Y}+\Delta_{X^{(s)}}))
\simeq \bigotimes ^s_{i=1} p_i^* \mathcal O_X(k(K_{X/Y}+\Delta)), 
\end{equation} and 
\begin{equation}
\label{eq3.2} 
(X^{(s)}, \Delta_{X^{(s)}}) \   \text{is semi-log canonical}. 
\end{equation}
We will check the isomorphism \eqref{eq3.1} and the statement \eqref{eq3.2}. 
We use induction on $s$. 
If $s=1$, then the isomorphism \eqref{eq3.1} and the statement \eqref{eq3.2} 
are obvious. 
By the induction hypothesis, 
we have 
\begin{equation}
\mathcal O_{X^{(s-1)}}(k(K_{X^{(s-1)}/Y}+\Delta_{X^{(s-1)}}))\simeq 
\bigotimes _{i=1}^{s-1} p_i^*\mathcal O_X(k(K_{X/Y}+\Delta)). 
\end{equation}
Therefore, it is sufficient to prove that 
\begin{equation}\label{eq3.4}
\begin{split}
&
\mathcal O_{X^{(s)}}(k(K_{X^{(s)}/Y}+\Delta_{X^{(s)}}))
\\ &
\simeq p_s^*\mathcal O_X(k(K_{X/Y}+\Delta))\otimes q_s^*\mathcal O_{X^{(s-1)}}
(k(K_{X^{(s-1)}/Y}+\Delta_{X^{(s-1)}})), 
\end{split}
\end{equation} 
where $q_s=(p_1, \cdots, p_{s-1}): X^{(s)}\to X^{(s-1)}$. 
The following commutative diagram
\begin{equation}\label{eq3.5}
\xymatrix{
X^{(s)}\ar[r]^{q_s}\ar[d]_{p_s}\ar[dr]^{f^{(s)}}&\ X^{(s-1)}\ar[d]^{f^{(s-1)}}\\
X\ar[r]_f & Y
}
\end{equation}
may be helpful. 
By the commutative diagram \eqref{eq3.5} 
and the induction hypothesis, we see that 
$X^{(s)}$ is demi-normal (see, for example, 
Step 1 in the proof of \cite[Lemma 2.12]{patakfalvi2}).  
By throwing out codimension at 
most two closed subsets, we may find a Zariski 
open subset $V\subset X^{(s)}$ such that 
$X|_{p_s(V)}$ 
and $X^{(s-1)}|_{q_s(V)}$ are Gorenstein and 
$k\Delta|_{p_i(V)}$ is Cartier for every $i$.  
By the flat base change theorem 
\cite[Chapter VII, Corollary 3.4]{hartshorne1} (see 
also \cite{conrad}), the isomorphism \eqref{eq3.4} holds over $V$. 
On the other hand, in \eqref{eq3.4}, 
the right hand side is locally free and the left hand side is reflexive. 
Therefore, we obtain the desired isomorphism \eqref{eq3.4} over $X^{(s)}$. 
This implies the isomorphism \eqref{eq3.1}. 
By \eqref{eq3.1},  we see that 
$\mathcal O_{X^{(s)}}(k(K_{X^{(s)}}+\Delta_{X^{(s)}}))$ is locally free 
and $f^{(s)}$-generated. 
By a special case of 
\cite[Lemma 2.12]{patakfalvi2} and Lemma \ref{f-lem2.13}, 
we see that $(X^{(s)}, \Delta_{X^{(s)}})$ is semi-log canonical. 
This is essentially the inversion of adjunction on semi-log canonicity 
(see \cite[Lemma 2.10 and Corollary 2.11]{patakfalvi2}, which 
is based on Kawakita's inversion of adjunction on log canonicity \cite{kawakita}).  
Moreover, we have: 
\begin{equation}\label{eq3.6}
f^{(s)}_* \mathcal O_{X^{(s)}}(k(K_{X^{(s)/Y}}+\Delta_{X^{(s)}})) 
\simeq \bigotimes ^{s} f_*\mathcal O_X(k(K_{X/Y}+\Delta)). 
\end{equation}
We will check the isomorphism \eqref{eq3.6}. 
We use induction on $s$. 
If $s=1$, then it is obvious. 
By \eqref{eq3.4}, 
we have 
\begin{equation*}
\begin{split}
&f^{(s)}_*\mathcal O_{X^{(s)}}(k(K_{X^{(s)}/Y}+\Delta_{X^{(s)}})) 
\\ 
& \simeq f_* {p_s}_* 
(p_s^*\mathcal O_X(k(K_{X/Y}+\Delta))\otimes 
 q_s^*\mathcal O_{X^{(s-1)}}
(k(K_{X^{(s-1)}/Y}+\Delta_{X^{(s-1)}})))
\\ 
& 
\simeq f_* (\mathcal O_X(k(K_{X/Y}+\Delta))\otimes {p_s}_* q_s^*\mathcal O_{X^{(s-1)}}
(k(K_{X^{(s-1)}/Y}+\Delta_{X^{(s-1)}})))
\\
&
\simeq f_* (\mathcal O_X(k(K_{X/Y}+\Delta))\otimes f^*f^{(s-1)}_*\mathcal O_{X^{(s-1)}}
(k(K_{X^{(s-1)}/Y}+\Delta_{X^{(s-1)}})))
\\ 
&\simeq f_*\mathcal O_X(k(K_{X/Y}+\Delta))\otimes f^{(s-1)}_* 
\mathcal O_{X^{(s-1)}}
(k(K_{X^{(s-1)}/Y}+\Delta_{X^{(s-1)}}))
\\
& \simeq \bigotimes ^s f_*\mathcal O_X(k(K_{X/Y}+\Delta))
\end{split} 
\end{equation*} 
by the projection formula and the flat base change theorem 
(see \cite[Chapter III, Proposition 9.3]{hartshorne2}). 
This is the desired isomorphism \eqref{eq3.6}. 
Note that $f_*\mathcal O_X(k(K_{X/Y}+\Delta))$ is locally free since 
$Y$ is a smooth projective curve. 

Let $L$ be an ample Cartier divisor 
on $Y$ such that $|L|$ is free. 
We put $M=kK_Y+2kL$. 
Then 
$$
f^{(s)}_* \mathcal O_{X^{(s)}}(k(K_{X^{(s)}/Y}+\Delta_{X^{(s)}}))\otimes 
\mathcal O_Y(M)
\simeq 
\left( \bigotimes ^{s} f_* \mathcal O_X(k(K_{X/Y}+\Delta))
\right)\otimes \mathcal O_Y(M)
$$
is globally generated by Corollary \ref{f-cor3.3}. 
Note that $M$ is independent of $s$. 
This implies that $f_*\mathcal O_X(k(K_{X/Y}+\Delta))$ is a nef locally free sheaf 
by Lemma \ref{f-lem3.4} below. 
\end{proof}

The following well-known lemma was already used in the proof of 
Theorem \ref{f-thm1.2}. 
We include it here for the reader's convenience. 

\begin{lem}\label{f-lem3.4}
Let $\mathcal E$ be 
a non-zero locally free sheaf of finite rank on a smooth 
projective variety $V$. 
Assume that 
there exists a line bundle $\mathcal M$ 
such that $\mathcal E^{\otimes s}\otimes \mathcal M$ 
is 
globally generated for every positive integer $s$. 
Then $\mathcal E$ is nef. 
\end{lem}
\begin{proof}
We put $\pi:W=\mathbb P_V(\mathcal E)\to V$ and 
$\mathcal O_W(1)=\mathcal O_{\mathbb P_V(\mathcal E)}(1)$. 
Since $\mathcal E^{\otimes s}\otimes \mathcal M$ is globally generated, 
$\mathrm{Sym}^s\mathcal E\otimes \mathcal M$ is 
also globally generated for every positive integer $s$. 
This implies that $\mathcal O_W(s)\otimes \pi^*\mathcal M$ is 
globally generated for every positive integer $s$. 
Thus, we obtain that $\mathcal O_W(1)$ is nef, equivalently, 
$\mathcal E$ is nef. 
\end{proof}

We sketch the proof of Theorem \ref{f-thm1.3} for 
the reader's convenience. 
For the details, see \cite{kollar-projectivity} and 
\cite{fujino-semipositive}. 

\begin{proof}[Sketch of Proof of Theorem \ref{f-thm1.3}]
As in \cite[Section 2]{kollar-projectivity}, 
we may assume that we start with a bounded 
moduli functor $\mathcal M$. 
We consider $(f: X\to C)\in \mathcal M (C)$, where $C$ is 
a smooth projective curve. 
Then, by \cite[Lemma 2.12]{patakfalvi2}, 
$X$ is a semi-log canonical variety. 
Since $\mathcal M$ is bounded, we can take a large and divisible positive 
integer 
$k$, which is independent of $C$, such that 
$\mathcal O_X(kK_{X/C})$ is locally free and $f$-generated. 
By Theorem \ref{f-thm1.2}, 
$f_*\mathcal O_X(klK_{X/C})$ is nef for every positive integer $l$. 
By Koll\'ar's projectivity criterion, 
this implies that every complete subspace of the moduli 
space of stable varieties is projective. 
For the details, see \cite[Sections 2 and 3]{kollar-projectivity}. 
\end{proof}

We close this section with a remark on {\em{slc morphisms}}. 

\begin{rem}\label{f-rem3.5}
In Theorem \ref{f-thm1.2}, 
(i) and (ii) are equivalent to 
the condition that $f:(X, \Delta)\to Y$ is an slc morphism 
(see \cite[Definition 6.11]{fujino-slc}) since 
$Y$ is a smooth curve. 
\end{rem}

\section{Proof of the basic semipositivity theorem}\label{f-sec4} 

In this section, we will give a proof of Theorem \ref{f-thm1.4}, 
which was first proved in \cite{fujino-semipositive},  
without using the theory of graded polarizable admissible 
variation of mixed Hodge structure (see \cite{fujino-fujisawa} and 
\cite{ffs}). 
Our approach to the basic semipositivity theorem (see Theorem \ref{f-thm1.4} and 
\cite[Theorem 1.9]{fujino-semipositive}) is arguably 
simpler than the arguments in \cite{fujino-semipositive}.  

Let us start with the following easy lemma, which is a variant of 
Lemma \ref{f-lem3.4}.  

\begin{lem}\label{f-lem4.1}
Let $\mathcal E$ be a non-zero locally free sheaf of 
finite rank on a smooth projective curve $C$. 
Let $\mathcal M$ be a fixed line bundle 
on $C$. 
Assume that $\mathcal E^{\otimes s}\otimes \mathcal M$ is generically 
globally generated for every positive integer $s$. 
Then $\mathcal E$ is nef. 
\end{lem}

\begin{proof}
Let $p:C'\to C$ be a finite morphism from a smooth projective 
curve $C'$. 
Let $\mathcal L$ be a quotient line bundle of 
$p^*\mathcal E$. 
Then we have a surjective morphism 
$$
p^*(\mathcal E^{\otimes s}\otimes \mathcal M)\to 
\mathcal L^{\otimes s}\otimes p^*\mathcal M. 
$$ 
By assumption, $p^*(\mathcal E^{\otimes s} \otimes \mathcal M)$ is 
generically globally generated for every positive integer 
$s$. 
Therefore, $\deg (\mathcal L^{\otimes s} \otimes p^*\mathcal M)\geq 0$ for 
every positive integer $s$. 
This means that $\deg \mathcal L\geq 0$. 
Therefore, $\mathcal E$ is nef. 
\end{proof}

We will prove Theorem \ref{f-thm1.4} by using 
the vanishing theorem for semi-log canonical pairs:~Theorem \ref{f-thm3.1}. 

\begin{proof}[Proof of Theorem \ref{f-thm1.4}] 
We will modify the proof of Theorem \ref{f-thm1.2}. 
Let $s$ be an arbitrary positive integer. 
Let $$X^{(s)}=\underbrace{X\times _C X\times _C \cdots \times _C X}_{s}$$ be 
the $s$-fold fiber product of $X$ over $C$ and let 
$f^{(s)}: X^{(s)}\to C$ be the induced natural map. 
Let $p_i$ be the $i$-th projection $X^{(s)}\to X$ for $1\leq i\leq s$. 
We put $D^{(s)}=\sum _{i=1}^s p_i^*D$. 
Then we have 
\begin{equation}\label{eq4.1}
\mathcal O_{X^{(s)}}(K_{X^{(s)}/C}+D^{(s)})\simeq 
\bigotimes _{i=1}^s p_i^* \mathcal O_X(K_{X/C}+D)
\end{equation} 
by the flat base change theorem. 
The isomorphism \eqref{eq4.1} can be checked 
similarly to the isomorphism \eqref{eq3.1} in the proof of 
Theorem \ref{f-thm1.2}. 
We note that the isomorphism \eqref{eq4.1} is equivalent to 
\begin{equation}\label{eq4.2} 
\mathcal O_{X^{(s)}}(K_{X^{(s)}/C})\simeq 
\bigotimes _{i=1}^s p_i^* \mathcal O_X(K_
{X/C})
\end{equation} 
by the definition of $D^{(s)}=\sum _{i=1}^s p_i^*D$. 
Moreover, 
we have 
\begin{equation}\label{eq4.3}
f^{(s)}_*\mathcal O_{X^{(s)}}(K_{X^{(s)}/C}+D^{(s)})\simeq 
\bigotimes ^s f_*\mathcal O_X(K_{X/C}+D). 
\end{equation}
Note that we can prove the isomorphism 
\eqref{eq4.3} similarly to the isomorphism \eqref{eq3.6} in the 
proof of Theorem \ref{f-thm1.2}. 
The following commutative diagram 
$$
\xymatrix{
X^{(s)}\ar[r]^{q_s}\ar[d]_{p_s}\ar[dr]^{f^{(s)}}&\ X^{(s-1)}\ar[d]^{f^{(s-1)}}\\
X\ar[r]_f & C, 
}
$$ 
where $q_s=(p_1, \cdots, p_{s-1}): X^{(s)}\to X^{(s-1)}$, 
and the isomorphism 
\begin{equation}
\mathcal O_{X^{(s)}}(K_{X^{(s)}/C}+D^{(s)})
\simeq p_s^*\mathcal O_X(K_{X/C}+D)\otimes q_s^*\mathcal O_{X^{(s-1)}}
(K_{X^{(s-1)}/C}+D^{(s-1)}), 
\end{equation} 
may be helpful. 

Let $U$ be a non-empty Zariski open subset of $C$ such that 
every stratum of $(X, D)|_{f^{-1}(U)}$ is smooth over $U$. 
Then we can directly see that 
$(X^{(s)}, D^{(s)})|_{f^{(s)-1}(U)}$ is 
semi-log canonical and that every irreducible component of $f^{(s)-1}(U)$ is smooth. 
Let $V$ be the largest Zariski open subset of 
$X^{(s)}$ such that 
$(V, D^{(s)}|_V)$ is a simple normal crossing pair. 
By construction, we see that $\codim _{f^{(s)-1}(U)}(f^{(s)-1}(U)\setminus V)\geq 2$. 
By the theorem of Bierstone--Vera Pacheco 
(see \cite[Theorem 1.4]{bierstone-verapacheco}), we have a projective 
surjective birational morphism 
$g: Z\to X^{(s)}$, 
which is given by a composite of blow-ups, 
an isomorphism over $V$, and maps $\Sing Z$ birationally 
onto the closure of $\Sing V$ in $X^{(s)}$, such that 
$\Exc(g)\cup \Supp g^*D^{(s)}$ is a simple normal crossing divisor on $Z$. 
Since $g$ is birational, 
we have a generically isomorphic injection $g_*\omega_Z\subset \omega_{X^{(s)}}$. 
Note that $Z$ and $X^{(s)}$ are both Gorenstein. 
Therefore, we have a generically isomorphic 
injection 
\begin{equation}
g_*\omega_Z(g^*D^{(s)})\subset \omega_{X^{(s)}}(D^{(s)}). 
\end{equation} 
We can take a reduced Weil divisor $\Delta_Z$ on $Z$ such that 
$\Supp \Delta_Z\subset \Exc(g)\cup \Supp g^*D^{(s)}$ and 
that 
\begin{equation}
g_*\omega_Z(\Delta_Z)\simeq \omega_{X^{(s)}}(D^{(s)})
\end{equation} 
holds on $f^{(s)-1}(U)$. 
We note that $(X^{(s)}, D^{(s)})|_{f^{(s)-1}(U)}$ is semi-log canonical. 
More precisely, we put 
\begin{equation}
\Delta_Z=\sum _i E_i +\sum _j F_j
\end{equation} 
where $E_i$ (resp.~$F_j$) runs over the irreducible 
components of $\Exc(g)$ (resp.~$g_*^{-1}D^{(s)}$) such that 
$f^{(s)}(E_i)\cap U\ne \emptyset$ 
(resp.~$f^{(s)}(F_j)\cap U\ne \emptyset$). 
Then $g_*\omega_Z(\Delta_Z)\simeq \omega_{X^{(s)}}(D^{(s)})$ holds on 
$f^{(s)-1}(U)$. 
By the definition of $\Delta_Z$, we have 
\begin{equation}
0\leq \Delta_Z\leq g^*D^{(s)}+\sum _i E_i. 
\end{equation}
Therefore, we have natural inclusions 
\begin{equation}\label{eq4.9}
g_*\omega_Z(\Delta_Z)\subset g_*\omega_Z\left(g^*D^{(s)}+\sum _i E_i\right)\subset 
\omega_{X^{(s)}}(D^{(s)}). 
\end{equation}
Note that $\codim _{X^{(s)}}(g(\sum _i E_i))\geq 2$ and that 
$\omega_{X^{(s)}}(D^{(s)})$ is locally free. 
By taking some suitable blow-ups of $Z$ outside $(f^{(s)}\circ g)^{-1}(U)$, 
if necessary, 
we may further assume that 
$\Delta_Z$ is a simple normal crossing divisor on $Z$, 
that is, $\Delta_Z$ is Cartier (see, for example, \cite[Lemma 2.11]{fujino-semipositive}). 
Anyway, we have a natural inclusion 
\begin{equation}
(f^{(s)}\circ g)_*\omega_{Z/C}(\Delta_Z)\subset f^{(s)}_* 
\omega_{X^{(s)}/C}(D^{(s)})\simeq 
\bigotimes ^s f_*\omega_{X/C}(D)   
\end{equation} by the inclusions 
\eqref{eq4.9} and the isomorphism \eqref{eq4.3}, 
which is an isomorphism on $U$. 
We put $\mathcal M=\omega_C\otimes \mathcal L^{\otimes 2}$, where 
$\mathcal L$ is an ample line bundle on $C$ such that 
$|\mathcal L|$ is free. 
Then $(f^{(s)}\circ g)_*\omega_{Z/C}(\Delta_Z)\otimes \mathcal M$ is globally 
generated by Corollary \ref{f-cor3.2}. 
Therefore, we obtain that 
\begin{equation}
\left(\bigotimes ^s f_*\omega_{X/C}(D)\right) \otimes \mathcal M
\end{equation} 
is generically globally generated for every positive integer $s$. 
By Lemma \ref{f-lem4.1}, 
we obtain that $f_*\omega_{X/C}(D)$ is nef. 
Anyway, we obtain that $f_*\omega_{X/C}(D)$ is nef without 
using the theory of variation of (mixed) Hodge structure. 
\end{proof}

\begin{rem}\label{f-rem4.2}
In the above proof of Theorem \ref{f-thm1.4}, 
we did not use the assumption that every stratum of $X$ is 
dominant onto $C$. 
It is sufficient to assume that 
$f:X\to C$ is flat for Theorem \ref{f-thm1.4} (see 
also \cite[Theorem 1.10]{fujino-semipositive}, which is more general than 
Theorem \ref{f-thm1.4}). 
\end{rem}

\end{document}